\documentclass[12pt]{amsart}
\usepackage{amssymb,xcolor}
\usepackage{amsmath}
\usepackage{stmaryrd}
\usepackage{graphicx}
\usepackage{mathrsfs}
\usepackage{url}
\usepackage{cite}
\usepackage{fullpage}
\usepackage[hidelinks]{hyperref}
\usepackage{array}
\usepackage{enumitem}

\usepackage{booktabs}

\newtheorem{thm}{Theorem}[section]
\newtheorem*{pgconj}{Prym--Green Conjecture}
\newtheorem{prop}[thm]{Proposition}
\newtheorem{lem}[thm]{Lemma}

\theoremstyle{definition}
\newtheorem{rem}[thm]{Remark}

\newcommand{\F}{\mathbb{F}}
\newcommand{\PP}{\mathbb{P}}

\newcommand{\CC}{\mathbb{C}}

\newcommand{\cO}{\mathcal{O}}
\newcommand{\cC}{\mathcal C}
\newcommand{\cL}{\mathcal L}
\newcommand{\cE}{\mathcal E}
\newcommand{\Aff}{\mathbb A}
\newcommand{\Rcal}{\mathcal{R}}

\DeclareMathOperator{\Sym}{Sym}
\DeclareMathOperator{\rank}{rank}
\DeclareMathOperator{\Tor}{Tor}
\DeclareMathOperator{\Ext}{Ext}

\DeclareMathOperator{\Proj}{Proj}
\DeclareMathOperator{\vspan}{span}
\DeclareMathOperator{\vdim}{dim}

\author{Sam Payne}
\email{sdpayne@umich.edu}
\author{Thomas Willwacher}
\email{thomas.willwacher@math.ethz.ch}

\thanks{
S.P. was supported in part by NSF grant DMS--2542134 and a Simons Fellowship.}

\title{The Prym--Green conjecture in even genera up to 30}

\begin{document}

\begin{abstract}
We prove the Prym--Green conjecture in every even genus from $20$ to $30$. The genus $24$ case disproves the suggestion of Chiodo--Eisenbud--Farkas--Schreyer that the conjecture might fail in every genus divisible by $8$. Our proof uses rational nodal curves with cyclic symmetry to decompose the relevant Koszul differential into a direct sum of smaller maps indexed by characters. We carry out the resulting rank computations over finite fields using the block Wiedemann algorithm.
\end{abstract}

\maketitle

\section{Introduction}\label{sec:intro}

Let $\Rcal_g$ be the moduli space of Prym curves over $\CC$. Its points are pairs $[C,\eta]$, where $C$ is a smooth projective curve of genus $g$ and $\eta$ is a nontrivial line bundle with $\eta^{\otimes 2}\cong\cO_C$. The paracanonical bundle $L=K_C\otimes\eta$ has $g-1$ independent global sections and, for a general pair of genus $g\geq 6$, defines an embedding $C\hookrightarrow\PP^{g-2}$.

Set $S=\Sym H^0(C,L)$. A minimal graded free resolution over $S$ is called \emph{natural} if, for each integer $d$, copies of $S(-d)$ occur in at most one term of the resolution.

\begin{pgconj}
For a general Prym curve $[C,\eta]\in\Rcal_g$ of genus $g\geq6$, the homogeneous coordinate ring of the embedding defined by $L=K_C\otimes\eta$ has a natural minimal graded free resolution over $S$.
\end{pgconj}

For even $g$, this is equivalent to vanishing of the Koszul cohomology group~\mbox{$K_{\frac g2-3,2}(C;L)$}.

\begin{thm}\label{thm:main}
The Prym--Green conjecture holds in all even genera \mbox{$20\leq g\leq 30$}.
\end{thm}

The conjecture was formulated by Farkas and Ludwig \cite[Conjecture~0.7]{FL}. Chiodo--Eisenbud--Farkas--Schreyer \cite{CEFS} verified it computationally in even genera up to $18$, except $8$ and $16$, where they found extra syzygies. A verification in genus $20$ was announced in \cite{CF,St17}. Farkas and Kemeny proved the conjecture in every odd genus \cite[Theorem~1.1]{FK}.

The computational evidence in genera $8$ and $16$ led Chiodo--Eisenbud--Farkas--Schreyer to suggest that the conjecture might fail whenever the genus is divisible by $8$, or perhaps whenever it is a power of $2$. Colombo--Farkas--Verra--Voisin proved that the conjecture fails in genus $8$ \cite{CFVV}. The expected failure in genus $16$ and the possibility of further exceptional genera remain open. The genus $24$ case of Theorem~\ref{thm:main} rules out the suggestion that the conjecture fails in every genus divisible by $8$.

\medskip

We work with rational nodal Prym curves over finite fields. Following \cite{CEFS}, we use Artinian reduction to translate the required vanishing into injectivity of a linear map. We choose curves of genus $g=2m$ with an action of the cyclic group $C_m$. We use a surjective multiplication map to eliminate variables and cyclic symmetry to split the remaining system into $m$ smaller maps indexed by characters. These maps are injective in genera $22$, $26$, and $30$. In genera $20$, $24$, and $28$, their direct sum has a two-dimensional kernel. We deform the curves and show that no nonzero kernel vector extends to second order. This proves the required vanishing on the generic fiber and establishes the conjecture for a general complex Prym curve.

The rank and kernel computations use the block Coppersmith--Wiedemann algorithm \cite{Coppersmith}, which Colombo--Frediani previously applied to Prym--Green computations on binary curves \cite[proof of Corollary~3.4]{CF}. Our implementation applies the maps directly from their defining formulas, without storing their matrices, and performs the main sequence computations on GPUs. The computations take approximately $7.5$ minutes in genus $24$ and $10.6$ days in genus $30$ on the hardware described in Section~\ref{sec:computations}. With our current implementation, the projected runtime in genus $32$ is approximately nine months.

\subsection*{Declaration of AI usage}
We used AI tools (ChatGPT Pro and Codex) for brainstorming, proofreading, and writing computer code.

\section{Artinian reduction and specialization}\label{sec:reduction}

We use the Artinian reduction and specialization method of \cite[Section~4]{CEFS}, that we recall here. Let $g=2m\geq8$, let $k$ be a field of characteristic different from $2$, and let $C$ be a geometrically integral projective curve of arithmetic genus $g$ with at worst nodal singularities. Let $\eta$ be a nontrivial line bundle with $\eta^{\otimes2}\cong\cO_C$, and put
\[
 L=\omega_C\otimes\eta,\qquad W=H^0(C,L),\qquad S=\Sym W,
 \qquad R=\bigoplus_{q\geq0}H^0(C,L^{\otimes q}).
\]
A degree-zero line bundle on $C$ with a nonzero section is trivial. Thus $H^0(C,\eta)=0$, and Serre duality and Riemann--Roch give
\[
 H^1(C,L^{\otimes q})=0,\qquad
 h^0(C,L^{\otimes q})=(2q-1)(g-1)\qquad(q\geq1).
\]
In particular, $\dim W=g-1$. We consider the Koszul cohomology groups
\[
 K_{p,q}(C;L):=\Tor^S_p(R,k)_{p+q}.
\]
Equivalently, $K_{p,q}(C;L)$ is the middle homology of the Koszul complex
\[
 \bigwedge^{p+1}W\otimes R_{q-1}\longrightarrow
 \bigwedge^pW\otimes R_q\longrightarrow
 \bigwedge^{p-1}W\otimes R_{q+1}.
\]
We seek curves satisfying
\begin{equation}\label{eq:critical}
 K_{m-3,2}(C;L)=0.
\end{equation}
For a general smooth Prym curve, this is equivalent to the vanishing in \cite[Conjecture~A]{CEFS}, as explained in Lemma~\ref{lem:pg-equivalence}. The section ring formulation lets us study this condition without assuming that $L$ is very ample or normally generated.

\subsection{The section ring and the pencil}

Choose $u,v\in W$ with no common zero on $C_{\bar k}$. The next lemma shows that $u,v$ form a regular sequence on $R$ and its canonical module.

\begin{lem}\label{lem:pg-section-ring}
The section ring $R$ is a finite Cohen--Macaulay $S$-module of dimension $2$. Its graded canonical module is
\[
 M:=\omega_R=\Ext^{g-3}_S(R,S(-(g-1)))
 \cong\bigoplus_{q\geq0}H^0(C,\omega_C\otimes L^{\otimes q}).
\]
The sections $u,v$ form a regular sequence on both $R$ and $M$. For $N=R$ or $M$, putting $T=S/(u,v)$ gives
\[
 \Tor^S_i(N,k)_j\cong\Tor^T_i\bigl(N/(u,v)N,k\bigr)_j.
\]
\end{lem}

\begin{proof}
The pencil gives an exact sequence
\[
 0\longrightarrow L^{\otimes(q-2)}
 \xrightarrow{\,(-v,u)\,}
 (L^{\otimes(q-1)})^{\oplus2}
 \xrightarrow{\,(u,v)\,}L^{\otimes q}\longrightarrow0.
\]
Taking global sections shows that every relation between $u$ and $v$ in $R$ is a Koszul relation. In particular, $vb\in uR$ implies $b\in uR$. Since $u$ is a nonzerodivisor, $u,v$ is an $R$-regular sequence. For $q\geq3$, the vanishing of $H^1(C,L^{\otimes(q-2)})$ also gives
\[
 R_q=uR_{q-1}+vR_{q-1}.
\]
Thus $R$ is finite over $k[u,v]$, and hence over $S$. Its Hilbert function has linear growth, so its dimension is $2$. Regularity of $u,v$ then shows that it is Cohen--Macaulay.

For each integer $q$, graded local duality identifies $(\omega_R)_q$ with $H^2_{S_+}(R)_{-q}^{\vee}$ \cite[Section~3.6]{BH}. The \v Cech complex for $u,v$ gives $H^2_{S_+}(R)_a=H^1(C,L^{\otimes a})$, and consequently
\[
 (\omega_R)_q\cong H^1(C,L^{\otimes(-q)})^\vee
 \cong H^0(C,\omega_C\otimes L^{\otimes q}).
\]
These identifications respect multiplication by sections. The groups on the right vanish for $q<0$, since the corresponding line bundle is $\eta$ for $q=-1$ and has negative degree for $q\leq-2$. The canonical module is maximal Cohen--Macaulay, so $u,v$ is also regular on $M$.

Finally, tensoring a minimal graded free resolution of $N$ with $S/(u,v)$ preserves exactness because $u,v$ is $N$-regular. The resulting resolution is still minimal, proving the assertion about $\Tor$.
\end{proof}

\subsection{The matrix}

Set
\[
 A=M/(u,v)M,\qquad V=W/\langle u,v\rangle,\qquad
 T=S/(u,v)=\Sym V.
\]
Thus $\dim V=g-3$. By Riemann--Roch,
\[
 H_M(t)=g+\sum_{q\geq1}(2q+1)(g-1)t^q,
 \qquad H_A(t)=(1-t)^2H_M(t)=g+(g-3)t+t^2.
\]
In particular $A_0$ and $A_1$ are the spaces used in Section~\ref{sec:artinian}:
\[
 A_0=H^0(C,\omega_C),\qquad
 A_1=H^0(C,\omega_C\otimes L)/(u,v)H^0(C,\omega_C).
\]

Artinian reduction applied directly to $R$ expresses the required vanishing as surjectivity of a Koszul differential on $R/(u,v)R$. In compatible bases, the corresponding differential on $A$ is represented by the transpose matrix. We use the canonical module to express the vanishing as injectivity, so that we can study kernel vectors under deformation in Section~\ref{sec:deformations}.

Multiplication $V\otimes A_0\to A_1$ determines a Koszul differential
\begin{equation}\label{eq:matrixmap}
 \Phi_g\colon\bigwedge^mV\otimes A_0\longrightarrow
 \bigwedge^{m-1}V\otimes A_1.
\end{equation}
We use the same notation for its matrix in chosen bases.

\begin{lem}\label{lem:pg-matrix}
The matrix $\Phi_g$ satisfies
\begin{equation}\label{eq:kernel}
 \dim\ker\Phi_g=\dim K_{m-3,2}(C;L).
\end{equation}
\end{lem}

\begin{proof}
The preceding term in the degree-$m$ part of the Koszul complex is $\bigwedge^{m+1}V\otimes A_{-1}=0$. Hence
\[
 \ker\Phi_g=\Tor^T_m(A,k)_m\cong\Tor^S_m(M,k)_m.
\]
Since $R$ is a Cohen--Macaulay $S$-module of codimension $g-3$, graded local duality gives $\Ext^a_S(R,S(-(g-1)))=0$ for $a\ne g-3$ \cite[Theorem~3.6.19]{BH}. Dualizing a minimal graded free resolution of $R$ with respect to $S(-(g-1))$ and reversing the resulting complex therefore gives a minimal resolution of $M$. The claim follows by taking $i=j=m$ in
\[
 \dim\Tor^S_i(M,k)_j
 =\dim\Tor^S_{g-3-i}(R,k)_{g-1-j}.\qedhere
\]
\end{proof}

Consequently, \eqref{eq:critical} holds if and only if $\Phi_g$ has full column rank. The dimensions are
\begin{equation}\label{eq:dimensions}
 \#\mathrm{rows}(\Phi_g)=(g-3)\binom{g-3}{m-1},\qquad
 \#\mathrm{columns}(\Phi_g)=g\binom{g-3}{m}.
\end{equation}
Each column has at most $m(g-3)$ nonzero entries, since the Koszul differential has $m$ terms and $\dim A_1=g-3$.

The following lemma relates the two critical Koszul cohomology groups to the resolution of the homogeneous coordinate ring.

\begin{lem}\label{lem:pg-equivalence}
Suppose \mbox{$u,v\in W$} have no common geometric zero. The following are equivalent:
\begin{enumerate}
\item \mbox{$K_{m-3,2}(C;L)=0$};
\item \mbox{$K_{m-2,1}(C;L)=0$};
\item $L$ is normally generated, hence very ample, and its homogeneous coordinate ring has a natural minimal free resolution.
\end{enumerate}
When these conditions hold, the resolution is pure.
\end{lem}

\begin{proof}
Put \mbox{$b_{p,q}=\dim K_{p,q}(C;L)$} and \mbox{$B=R/(u,v)R$}. The Hilbert series of $B$ is
\[
 H_B(t)=(1-t)^2H_R(t)=1+(g-3)t+gt^2.
\]
In particular \mbox{$B_1=V$} and \mbox{$B_q=0$} for \mbox{$q\geq3$}. By Artinian reduction, \mbox{$b_{p,q}=0$} for \mbox{$q\geq3$}. Also \mbox{$b_{p,0}=0$} for \mbox{$p>0$} because \mbox{$R_0=k$} and \mbox{$R_1=W$}. Comparing the coefficient of $t^{m-1}$ in
\[
 \sum_{p,q}(-1)^p b_{p,q}t^{p+q}
 =(1-t)^{g-3}\bigl(1+(g-3)t+gt^2\bigr)
\]
gives \mbox{$b_{m-2,1}=b_{m-3,2}$}, since that coefficient is zero. This proves the equivalence of (1) and (2).

Assume (1), and put \mbox{$p=m-3$}. Artinian reduction gives
\[
 K_{p,2}(C;L)\cong\Tor^T_p(B,k)_{p+2}.
\]
This group surjects onto
\[
 \bigwedge^pV\otimes\bigl(B_2/(V B_1)\bigr),
\]
because the outgoing Koszul differential is zero, while the incoming image lies in \mbox{$\bigwedge^pV\otimes V B_1$}. Since \mbox{$\bigwedge^pV\neq0$}, we obtain \mbox{$B_2=V B_1$}. Lifting generators from $B$ to $R$ now shows that $R$ is generated by $R_1$. Thus $L$ is normally generated. Since $L$ is ample and \mbox{$C=\Proj R$}, it also defines an embedding with homogeneous coordinate ring $R$.

Minimality and exactness imply that \mbox{$b_{i,1}=0$} forces \mbox{$b_{i+1,1}=0$} for \mbox{$i\geq1$} because a generator in the latter term could map only to the degree-\mbox{$i+1$} summands of the preceding term. Applying the same argument to the dual resolution shows that \mbox{$b_{i,2}=0$} implies \mbox{$b_{i-1,2}=0$} for \mbox{$1\leq i\leq g-3$}. Hence the two critical vanishings imply \mbox{$b_{p,1}=0$} for \mbox{$p\geq m-2$} and \mbox{$b_{p,2}=0$} for \mbox{$p\leq m-3$}. The resolution is therefore natural, and in fact pure. Conversely, naturality and \mbox{$b_{m-2,1}=b_{m-3,2}$} force both critical groups to vanish.
\end{proof}

\subsection{Specialization}

Over a finite field $k$ of odd characteristic, choose $2g$ distinct points
\[
 P_1,Q_1,\ldots,P_g,Q_g\in\PP^1(k)
\]
and identify $P_i$ with $Q_i$ for each $i$. The resulting curve $C_0$ is geometrically integral and stable, with $g$ nodes. Gluing the trivial line bundle on its normalization with multipliers in $\{1,-1\}$, at least one of which is $-1$, gives a nontrivial line bundle $\eta_0$ of order two. The section spaces for $L_0=\omega_{C_0}\otimes\eta_0$ and for the canonical module are spaces of binary forms subject to linear gluing conditions. Their bases and multiplication maps can therefore be computed over $k$.

\begin{lem}\label{lem:pg-openness}
Let $\pi\colon\cC\to D$ be a proper flat family of geometrically integral nodal curves of arithmetic genus $g$ over a noetherian base, and let $\cE$ be a line bundle that is nontrivial of order two on every geometric fiber. Set $\cL=\omega_{\cC/D}\otimes\cE$. Then $\mathcal V_q=\pi_*\cL^{\otimes q}$ is locally free and commutes with base change for $q\geq1$. The loci where $L_d$ is globally generated and where $K_{m-3,2}(C_d;L_d)=0$ are open in $D$.
\end{lem}

\begin{proof}
Cohomology and base change, together with the vanishing of $H^1(C_d,L_d^{\otimes q})$ for $q\geq1$, give the assertion about $\mathcal V_q$. For $p=m-3$, the complex of vector bundles
\[
 \bigwedge^{p+1}\mathcal V_1\otimes\mathcal V_1\longrightarrow
 \bigwedge^p\mathcal V_1\otimes\mathcal V_2\longrightarrow
 \bigwedge^{p-1}\mathcal V_1\otimes\mathcal V_3
\]
computes $K_{p,2}(C_d;L_d)$ on each fiber. Its middle homology has upper semicontinuous dimension, because the ranks of the differentials are lower semicontinuous. Thus $K_{p,2}=0$ is an open condition in such a family. Failure of global generation is the image of the support of $\operatorname{coker}(\pi^*\mathcal V_1\to\cL)$, which is closed because $\pi$ is proper.
\end{proof}

\begin{prop}\label{prop:specialization}
Suppose that $C_0,\eta_0$ are as above, that \mbox{$u,v\in H^0(C_0,L_0)$} have no common geometric zero, and that the resulting matrix $\Phi_g$ has full column rank. Then \mbox{$K_{m-3,2}(C;\omega_C\otimes\eta)=0$} for a general complex Prym curve \mbox{$[C,\eta]\in\Rcal_g$}. In particular, the Prym--Green conjecture holds in genus $g$.
\end{prop}

\begin{proof}
We first lift the nodal pair to characteristic zero and then smooth it, applying Lemma~\ref{lem:pg-openness} at each step.

Choose a discrete valuation ring $\Lambda$, obtained by localizing the ring of integers of a number field, with residue field $k$ and fraction field $F$. Lift the $2g$ marked points to sections of $\PP^1_{\Lambda}$. They are pairwise disjoint, since their closed-fiber points are distinct. The pinching construction \cite[Theorem~5.4]{Fer} identifies these sections in the prescribed pairs. The resulting family has rational $g$-nodal fibers. It is flat over $\Lambda$ because its structure sheaf embeds in the pushforward of $\cO_{\PP^1_\Lambda}$ and is therefore torsion-free over the DVR. A positive-degree line bundle on $\PP^1_\Lambda$ descends after choosing identifications at the paired sections and is relatively ample, so the family is projective. Gluing the trivial line bundle with the same signs gives a line bundle of order two, nontrivial on both geometric fibers. Since $\Phi_g$ has full column rank, \eqref{eq:kernel} gives $K_{m-3,2}(C_0;L_0)=0$, and $L_0$ is globally generated by hypothesis. Both properties persist on the generic fiber. Extending scalars along an embedding $F\hookrightarrow\CC$ therefore produces a complex nodal Prym pair $(C_1,\eta_1)$ with $K_{m-3,2}(C_1;\omega_{C_1}\otimes\eta_1)=0$.

We next smooth the complex nodal pair $(C_1,\eta_1)$. The stack of stable curves equipped with a line bundle of exact order two and a trivialization of its square is \'etale over $\overline{\mathcal M}_g$ \cite[Section~1.2.1, final paragraph, p.~608]{CF17}. Consequently, locally on a deformation space of $C_1$, the line bundle $\eta_1$ and its trivialization extend. Since stable curves are smoothable, we may choose a one-parameter smoothing of the pair $(C_1,\eta_1)$. After shrinking the base, all fibers are geometrically integral. The same openness argument then gives a smooth Prym curve with a globally generated paracanonical bundle and the required vanishing.

Finally, the same section of \cite{CF17} states that the smooth Prym stack is connected and \'etale over $\mathcal M_g$, hence smooth and irreducible. Thus both properties hold on a dense open subset of $\Rcal_g$. For such a complex pair, choose a pencil in $H^0(C,L)$ with no common zero. Lemma~\ref{lem:pg-equivalence} then proves the Prym--Green conjecture.
\end{proof}

\section{Parameter choices and further simplifications}\label{sec:parameters}

We use a surjective multiplication map to eliminate variables from the equations defining $\ker\Phi_g$, then use cyclic symmetry to split the remaining system into blocks indexed by characters, making the computations feasible in genera $28$ and $30$.

\subsection{The curves, line bundles and sections}
Fix an odd prime number $p$. We will work over the field $\F_p$. To define the curve $C$, we choose $2g$ distinct numbers
\[
P_1,Q_1,\dots,P_g,Q_g\in \{0,\dots, p-1\},
\]
that we consider as points in $\PP^1(\F_p)$. The curve $C=C(P_1,Q_1,\dots,P_g,Q_g)$ is then defined by identifying $P_j$ and $Q_j$ for $j=1,\dots, g$. A line bundle on $C$ may be defined by giving the datum of a degree $d$ and elements $a_1,\dots,a_g\in \F_p^\times$. The global sections of this bundle are then polynomials $f\in \F_p[z]$ of degree $\leq d$ such that $f(P_j)=a_j f(Q_j)$ for $j=1,\dots,g$.

We define the line bundle $\eta$ by setting $d=0$ and $a_j=-1$, $j=1,\dots,g$. Put
\[
D(z) = \prod_{j=1}^g (z-P_j)(z-Q_j).
\]
Then the canonical bundle $K_C$ may be defined by setting $d=2g-2$ and
\[
a_j = - \frac{D'(P_j)}{D'(Q_j)}.
\]
More precisely, we understand a polynomial $f(z)$ as the section
\[
\frac{f(z)}{D(z)} dz.
\]
Then the compatibility relations state that the residues of these forms at $P_j$ and $Q_j$ have to be opposite, that is
\[
\frac{f(P_j)}{D'(P_j)} = - \frac{f(Q_j)}{D'(Q_j)}
\ \Leftrightarrow \
f(P_j) = - \frac{D'(P_j)}{D'(Q_j)} f(Q_j).
\]

We define the line bundle $L=K_C \otimes \eta$ -- this has $d=2g-2$ and $a_j = \frac{D'(P_j)}{D'(Q_j)}$. Let $W=H^0(C; L)\cong \F_p^{g-1}$ be the $\F_p$-vector space of global sections. Concretely, we consider elements of $W$ as polynomials $f$ of degree $\leq 2g-2$ such that
\begin{equation}\label{equ:gluing}
f(P_j) = \frac{D'(P_j)}{D'(Q_j)} f(Q_j) \quad\text{for $j=1,\dots,g$}.
\end{equation}

\subsection{\texorpdfstring{$C_m$}{Cm} Symmetry and choice of the \texorpdfstring{$P_j,Q_j$}{Pj,Qj}}
In the following let $g=2m$ be even. We choose the primes $p$ and the points $P_j,Q_j$ so that the resulting curve $C$ has an action of the cyclic group $C_m$. More precisely, let $\zeta\in \F_p$ be a primitive $m$th root of unity. Then we pick only $P, Q, P',Q'\in \F_p$ and set
\begin{align*}
P_j &= \zeta^{j-1} P & Q_j &= \zeta^{j-1} Q & \text{for $j=1,\dots,m$},
\\
P_j &= \zeta^{j-1} P' & Q_j &= \zeta^{j-1} Q' & \text{for $j=m+1,\dots,g$},
\end{align*}
with the implicit restriction on $P,Q,P',Q'$ that the $P_j$ and $Q_j$ thus constructed are all distinct. The automorphism $z\mapsto\zeta z$ preserves the node identifications and induces an action of the cyclic group $C_m$ on $C$. The action extends to $W=H^0(C;L)$. Representing elements of $W$ as polynomials $f$ of degree $\leq 2g-2$ as above, the action of the generator is explicitly
\[
f(z) \mapsto  \zeta f(\zeta z).
\]

Varying $P,Q,P',Q'$ gives a family of curves. For all six genera, we use
\begin{align*}
P&=1 & Q&=2 & P'&=4 & Q'&=10.
\end{align*}
The following table lists the corresponding choices of $p$ and $\zeta$.
\begin{center}
\begin{tabular}{c|c|c|}
$g$ & $p$ & $\zeta$ \\ \hline 20 & 661 & 190 \\ 22 & 661 & 9 \\ 24 & 1009 & 160 \\ 26 & 1093 & 11 \\ 28 & 1009 & 74 \\ 30 & 1051 & 136
\end{tabular}
\end{center}

\subsection{Artinian reduction}\label{sec:artinian}
Next we pick two sections $u,v\in W$ that satisfy the following conditions:
\begin{enumerate}
\item $u,v$ have no common zero on $\PP^1_{\overline{\F}_p}$.
\item $u,v$ must be $C_m$-eigenvectors. More precisely, let
\[
W_c = W\cap \vspan\{z^j \mid j+1\equiv c \mod m\}
\]
be the $\zeta^c$-eigenspace for the action of the $C_m$-generator. Then we require that $u\in W_1$ and $v\in W_{m-1}$.
\end{enumerate}

We first compute bases of the two-dimensional spaces $W_c$ ($c=1,m-1$) by solving \eqref{equ:gluing} for polynomials in $\vspan\{z^j \mid j+1\equiv c \mod m\}$. We then test pairs $u\in W_1$, $v\in W_{m-1}$ until we find a pair with no common zero. We check this condition by testing that $\gcd(u,v)=1$ and that $([z^{2g-2}]u,[z^{2g-2}]v)\neq (0,0)$, which excludes a common zero at $\infty$.

Having found $u,v$ we define as before
\begin{align*}
    V&=W/\langle u,v\rangle &\text{(dimension $g-3$)} \\
    A_0&= H^0(C;\omega_C) &\text{(dimension $g$)} \\
    A_1&= H^0(C;\omega_C\otimes L)/ \langle u,v\rangle A_0 &\text{(dimension $g-3$)}.
\end{align*}
For $x\in V$ we define the multiplication maps ($(g-3)\times g$-matrices)
\[
\mu_x\colon A_0 \to A_1 \quad a\mapsto xa.
\]
These matrices can be pre-computed, for $x$ ranging over a basis of $V$, and
\[
\Phi_g \colon \bigwedge^m V \otimes A_0 \to \bigwedge^{m-1} V \otimes A_1
\]
is fully expressed through the matrices $\mu_x$.

\subsection{A further reduction}\label{sec:further-reduction}
Let $w\in V$ be an eigenvector (for the $C_m$-generator $\zeta\cdot$) such that $\mu_w\colon A_0\to A_1$ is surjective. Write $\zeta\cdot w=\zeta^q w$.

For a $C_m$-representation $E$, we write $E_c$ for its $\zeta^c$-eigenspace, with $c$ taken modulo $m$.

Pick a $C_m$-stable complement $U$ of $w$ in $V$, i.e., $V=\langle w\rangle\oplus U$. Pick a right inverse $R\colon A_1\to A_0$ to $\mu_w$ such that $R\colon A_{1,c}\to A_{0,c-q}$. We have $\mu_w R= \mathit{id}_{A_1}$. Let $\iota\colon K_w:=\ker \mu_w\to A_0$ be the inclusion. Note that $\vdim K_w=3$.

Also, split
\[
\bigwedge^m V \otimes A_0
\cong
\bigwedge^m U \otimes A_0
\oplus
\langle w\rangle \otimes \bigwedge^{m-1} U \otimes A_0,
\]
and similarly
\[
\bigwedge^{m-1} V \otimes A_1
=
\bigwedge^{m-1} U \otimes A_1
\oplus
\langle w\rangle \otimes \bigwedge^{m-2} U \otimes A_1.
\]
Let
\[
D_k\colon\bigwedge^k U\otimes A_0\longrightarrow\bigwedge^{k-1}U\otimes A_1
\]
be the Koszul differential determined by the multiplication maps $\mu_s$ for $s\in U$. Then the equation $\Phi_g(x+w\wedge y)=0$ decomposes into the two equations
\begin{align*}
    D_m x + \mu_w y &=0 & D_{m-1}y=0.
\end{align*}
Write $y=\iota z+R y'$ with $y'\in \bigwedge^{m-1} U \otimes A_1$, $z\in \bigwedge^{m-1} U \otimes K_w$, then the first equation can be solved for $y'$ and we obtain
\[
y =\iota z -RD_m x.
\]
Inserting this into the second equation we obtain the single equation for $x,z$
\begin{equation}\label{equ:Fdef}
F(x,z) := D_{m-1}(\iota z -RD_m x ) =0.
\end{equation}
Hence
\begin{equation}\label{equ:Fdef2}
\begin{aligned}
F \colon \bigwedge^{m} U \otimes A_0 \oplus \bigwedge^{m-1} U \otimes K_w &\to \bigwedge^{m-2} U \otimes A_1
\end{aligned}
\end{equation}
satisfies
\[
\ker F \cong \ker \Phi_g .
\]
The matrix of $F$ has significantly smaller size
\[
\binom{g-4}{m-2}(g-3) \times \left( \binom{g-4}{m} g + 3 \binom{g-4}{m-1} \right).
\]

\subsection{Splitting into \texorpdfstring{$C_m$}{Cm}-isotypical components}
We next use the $C_m$-action to split the matrix of $F$ according to the isotypical components of this action. Concretely, suppose that $\zeta\cdot w=\zeta^q w$ -- in our program we choose $q=2$. The map $F$ splits into the components
\[
F_c \colon
\left( \bigwedge^{m} U \otimes A_0\right)_c \oplus \left(\bigwedge^{m-1} U \otimes K_w\right)_{c-q}
\to
\left( \bigwedge^{m-2} U \otimes A_1\right)_{c-q}.
\]
This decomposition replaces the matrix of $F$ by $m$ smaller matrices.

These curves form a special family, so the Prym--Green conjecture does not imply that $F$ is injective for our choices. For odd $m$ in the range considered, our computations show that the above choice suffices and $F$ has full column rank. More specifically, computation over $\F_p$ verifies the following proposition.
\begin{prop}\label{prop:F comp result}
    For the choices of $P_j,Q_j$ and $p$ detailed above and a suitable choice of $u,v,w$, we have:
    \begin{itemize}
    \item For $g=22,26,30$ the matrices $F_c$ have full column rank for each $c=0,\dots,m-1$.
    \item For $g=20,24,28$ the matrices $F_c$ have full column rank for each $c=1,\dots,m-1$, while $F_0$ has a two-dimensional kernel.
    \end{itemize}
\end{prop}
In the latter case, the program also computes a basis of the kernel of $F_0$. Other choices of the parameters, points and group action did not yield injective matrices for even $m$ with comparable computation times. For even $m$, we therefore deform the chosen curves as described in the next subsection.

\subsection{Deformations for even \texorpdfstring{$m$}{m}}\label{sec:deformations}
Let $m$ be even. We will consider the formal deformation
\begin{align}\label{equ:PtQt}
P_j(t) &= P_j +t \dot P_j & Q_j(t) &= Q_j +t \dot Q_j,
\end{align}
with $\dot P_j$, $\dot Q_j$ in $\F_p$ and $t$ a formal parameter. The same formulas define a curve over $\F_p(t)$, whose base change to $\F_p(\!(t)\!)$ we use below. For $g=20,24,28$, we use the following deformation:
\begin{align}\label{equ:dotPdotQ}
\dot P_j &= \begin{cases}
    \zeta^{2j-2} & \text{for } j=1,\dots,m \\
    0 & \text{for } j=m+1,\dots,g
\end{cases},
&
\dot Q_j &= \begin{cases}
    1 & \text{for } j=1,\dots,m \\
    0 & \text{for } j=m+1,\dots,g
\end{cases}.
\end{align}
We make the auxiliary choices in the construction of $F$ over $\F_p\llbracket t\rrbracket$, starting from the choices at $t=0$. The section spaces form free modules and commute with reduction modulo $t$. We may therefore lift $u,v,w$ and the chosen bases. The pencil remains without a common geometric zero, and $\mu_w$ remains surjective, since both conditions are open and hold on the closed fiber.
In lifted bases, the chosen right inverse $R(0)$ extends to
\[
 R(t)=R(0)\bigl(\mu_w(t)R(0)\bigr)^{-1},
\]
since the matrix in parentheses is the identity modulo $t$. Complements and kernel bases can likewise be lifted. Thus the reduction of Section~\ref{sec:further-reduction} applies over $\F_p(\!(t)\!)$, and
\[
 \ker F(t)\cong\ker\Phi_g(t).
\]
The notation $F^{(i)}$ below refers to the Taylor coefficients of this matrix for the actual family of Prym curves. The specialization argument does not require the chosen formal bases to have rational entries in $t$.

The following lemma reduces the finite-field construction to injectivity of $F(t)$ over $\F_p(\!(t)\!)$.

\begin{lem}\label{lem:finite-field-specialization}
Suppose $F(t)$ has zero kernel over $\F_p(\!(t)\!)$. Then there are a finite extension $k/\F_p$ and $a\in k$ such that the points $P_j(a),Q_j(a)$ are pairwise distinct and the corresponding Prym curve admits sections $u_a,v_a$ with no common geometric zero for which $\Phi_g$ has zero kernel.
\end{lem}

\begin{proof}
The moving marked points define a family of rational $g$-nodal curves $\pi\colon\cC\to U$ over an open neighborhood $U$ of $0$ in $\Aff^1_{\F_p}$. To construct this family, remove the finitely many parameters at which two marked points coincide, and identify the corresponding sections in pairs. Use the same gluing signs to define the relative line bundle $\cE$ of order two, and put $\cL=\omega_{\cC/U}\otimes\cE$.

By Lemma~\ref{lem:pg-openness}, the sheaves $\pi_*\cL^{\otimes q}$ are locally free and commute with base change for $q\geq1$. After shrinking $U$ around $0$, lift the chosen sections $u(0),v(0)$ to sections $u,v$ of $\pi_*\cL$. Their common-zero locus has closed image in $U$ because $\pi$ is proper. This image does not contain $0$, so we may shrink $U$ further to ensure that $u,v$ have no common geometric zero on any fiber.

By the reduction above and Lemma~\ref{lem:pg-matrix}, injectivity of $F(t)$ gives
\[
 K_{m-3,2}(C_{\F_p(\!(t)\!)};L_{\F_p(\!(t)\!)})=0.
\]
The Koszul complex commutes with the field extension $\F_p(t)\subset\F_p(\!(t)\!)$. Faithful flatness therefore gives the same vanishing on the generic fiber over $U$. By Lemma~\ref{lem:pg-openness}, this vanishing holds on a nonempty open subset $U^\circ\subset U$. Choose a closed point of $U^\circ$. Its residue field $k$ is finite over $\F_p$, and its parameter value is an element $a\in k$. The marked points are distinct and the specialized pencil has no common geometric zero. Lemma~\ref{lem:pg-matrix} therefore gives $\ker\Phi_g=0$.
\end{proof}

By multiplying by a power of $t$ we may always assume that (potential) kernel vectors of $F(t)$ have the form
\[
v(t) = v_0+tv_1+t^2v_2+\cdots  \quad \text{with $v_0\neq 0$}.
\]
In the following let us introduce the notation
\[
F(t) = F^{(0)} + t F^{(1)} + t^2 F^{(2)} + O(t^3),
\]
with all block indices below referring to the source and target decompositions of $F_c$ at $t=0$, including the shifts by $q$. Up to quadratic order the obstruction problem then asks for a solution of the linear system
\begin{align*}
    F^{(0)} v_0 &=0
    \\
    F^{(1)} v_0 + F^{(0)} v_1 &=0
    \\
    F^{(2)} v_0 + F^{(1)} v_1 + F^{(0)} v_2 &=0.
\end{align*}

In particular, $v_0$ is in the kernel of $F^{(0)}$ defined over $\F_p$ as studied above, and hence a linear combination of the two kernel basis vectors found. We will show that this system has no solution $(v_0,v_1,v_2)$ with $v_0\neq0$. Thus no nonzero element of $\ker F^{(0)}$ extends to a kernel vector modulo $t^3$.

\subsubsection{$C_m$-equivariant analysis for adapted choices}
To analyze how the deformation defined by \eqref{equ:PtQt} and \eqref{equ:dotPdotQ} mixes the $C_m$-isotypical components, we introduce a two-parameter family
\begin{align*}
P_j(t_{-1},t_1) &= \begin{cases}
    \zeta^{j-1} + t_{-1} \zeta^{2j-2} & \text{for } j=1,\dots,m \\
    4 \zeta^{j-1} & \text{for } j=m+1,\dots,g
\end{cases},
\\
Q_j(t_{-1},t_1) &= \begin{cases}
    2\zeta^{j-1} + t_1 & \text{for } j=1,\dots,m \\
    10 \zeta^{j-1} & \text{for } j=m+1,\dots,g
\end{cases}.
\end{align*}
This configuration is $C_m$-invariant if we extend the action of $C_m$ to the parameter space by the assignment
\begin{align*}
    \zeta \cdot t_{\pm 1} = \zeta^{\pm 1} t_{\pm 1}.
\end{align*}
We make the auxiliary choices in the construction of $F$ compatibly with this $C_m$-action. We extend $u,v,w$ to sections $u(t_{-1},t_1),v(t_{-1},t_1),w(t_{-1},t_1)$ of characters $1,m-1,2$, respectively, and choose the complement $U$, right inverse $R$, and bases of $A_0,A_1,K_w$ to respect the action. The resulting matrix has an expansion
\[
F(t_{-1},t_1)
= \sum_{a,b\geq 0} F^{(a,b)} t_{-1}^a t_1^b
\]
in which $F^{(a,b)}$ transforms according to the character $a-b$ under the $C_m$-action. Equivalently, $F^{(a,b)}$ sends the $c$-isotypical component to the $c+a-b$-isotypical component.

The substitution
\begin{equation}\label{equ:t sub}
t_{-1}=t_1=t
\end{equation}
recovers the original one-parameter deformation and breaks the $C_m$-symmetry. The Taylor coefficients of $F(t)$ satisfy
\[
F^{(j)} = \sum_{a,b\geq 0 \atop a+b=j} F^{(a,b)}.
\]
In particular, we obtain the following conclusions:
\begin{enumerate}
    \item $F^{(0)}$ only has the diagonal blocks $F_c^{(0)}=F_{c,c}^{(0)}$ relative to the $C_m$-isotypical decomposition.
    \item $F^{(1)}$ only has blocks $F_{c\pm 1,c}^{(1)}$ that send the $c$-isotypical component to that of character $c\pm 1$.
\end{enumerate}

\begin{rem}
We remark that our program internally does not actually compute two-parameter families explicitly, but only one parameter families modulo $t^3$. However, it makes choices that produce data that are compatible with the $C_m$-action in the above sense. Below it will suffice to know only the above properties (1) and (2) of $F(t)$, which the program enforces, respectively checks. The analysis of this section shows that there exist such choices, but it should not be seen as a direct algorithmic description of the program.
\end{rem}

\subsubsection{Vanishing of the first-order obstruction}
The first-order obstruction equation reads
\[
 F^{(1)} v_0 + F^{(0)} v_1=0.
\]
Let $K$ be the two-column matrix containing the basis vectors of $\ker F^{(0)}$, so that $F^{(0)}K=0$. Then we need to solve the system
\begin{equation}\label{equ:SK}
F^{(0)} S = - F^{(1)} K
\end{equation}
for the two-column matrix $S$. We analyze \eqref{equ:SK} using these source and target decompositions into $C_m$-isotypical components. The matrix $F^{(0)}$ is block diagonal with blocks $F_c^{(0)}$ indexed by $c$. All these have full column rank except for $F_0^{(0)}$, by Proposition~\ref{prop:F comp result}. Hence the only nonzero block of $K$ is the block $K_0$ indexed by $0$.
Write again $F_{a,b}^{(i)}$ for the block from source block $b$ to target block $a$. By the analysis of the previous subsection, the block $F_{0,0}^{(1)}$ vanishes, and the only possibly nonzero blocks of $F^{(1)}$ with source block $0$ are $F_{\pm 1,0}^{(1)}$. (We assume here that the choices in our construction have been made so as to be compatible with the $C_m$-action, as described in the previous subsection.) We therefore solve the following two independent systems over $\F_p$:
\begin{align}\label{equ:SK decomp}
 F_1^{(0)} S_1  &= - F_{1,0}^{(1)} K_0
 &
 F_{-1}^{(0)} S_{-1}  &= - F_{-1,0}^{(1)} K_0.
\end{align}

We find that in all cases considered ($g=20,24,28$) this system is solvable and hence there is no first-order obstruction. We also remark that the solution for $S$ is not unique, but unique up to adding elements in $\ker F^{(0)}$ to its columns.

\subsubsection{Second-order obstruction}
We write the second-order obstruction equation
\[
F^{(2)} K \alpha + F^{(1)} (S\alpha +K \beta) + F^{(0)} v_2 =0,
\]
where $\alpha,\beta\in\F_p^2$ are to be solved for along with $v_2$. We claim that this linear equation has no solution with $\alpha\neq 0$. First simplify by noting that by \eqref{equ:SK}
\[
F^{(1)} K \beta = - F^{(0)} S \beta,
\]
so that, replacing $v_2$ by $v_2-S\beta$, it suffices to show that
\begin{equation}\label{equ:Z full}
Z \alpha + F^{(0)} v_2 =0
\end{equation}
has no solution for $\alpha\neq 0, v_2$, where
\[
Z:= F^{(2)} K + F^{(1)}S.
\]
Projecting \eqref{equ:Z full} to target block $0$ and using again that $F^{(0)}$ is block-diagonal, it suffices to show that
\begin{equation}
Z_0 \alpha + F_0^{(0)} v_{2,0} =0
\end{equation}
has no solution for $\alpha, v_{2,0}$ with $\alpha\neq 0$, where $Z_0$ is the component of $Z$ in target block $0$. To check this, we may verify that
\[
\rank [F_0^{(0)} \ Z_0 ] = 2+\rank F_0^{(0)}.
\]
Since $F_0^{(0)}$ has a two-dimensional kernel, an equivalent formulation is as follows. Let $i$, $j$ be two indices such that the $i$th and $j$th rows of the kernel basis $K_0$ of $F_0^{(0)}$ form an invertible $2\times 2$ matrix. Let $F_0^{(0),red}$ be obtained by omitting the $i$th and $j$th columns of $F_0^{(0)}$, so that the images of $F_0^{(0),red}$ and $F_0^{(0)}$ are the same. Then the condition to be checked is that the \emph{replacement matrix}
\[
F_{rep}:=[F_0^{(0),red} \ Z_0 ]
\]
has full column rank. Note that $F_{rep}$ has the same size as $F_0^{(0)}$. Also note that, using our explicit isotypical decompositions, we have the formula
\[
Z_0 = F_{0,0}^{(2)} K_0 + F_{0,1}^{(1)}S_1 + F_{0,-1}^{(1)}S_{-1}.
\]

So the overall verification algorithm for $m$ even is as follows:
\begin{enumerate}
    \item Compute rank lower bounds showing that $F_c^{(0)}$ has full column rank for $c\neq0$ and that $F_0^{(0)}$ has rank at least $n_0-2$, where $n_0$ is its number of columns.
    \item Compute two independent columns $K_0$ and verify $F_0^{(0)}K_0=0$. Together with the lower bound, this proves that $K_0$ is a kernel basis.
    \item Compute $S_{\pm1}$ and verify the solution identities \eqref{equ:SK decomp}.
    \item Choose the deleted columns using an invertible $2\times2$ row minor of $K_0$, form the replacement matrix $F_{rep}$, and show that $F_{rep}$ has full column rank by computing a rank lower bound.
\end{enumerate}
In all cases considered, $F_{rep}$ has full column rank. The second-order obstruction therefore excludes every nonzero element of $\ker F^{(0)}$, and $F(t)$ is injective over $\F_p(\!(t)\!)$.

\section{Implementation and computational details}\label{sec:computations}
\subsection{Block Coppersmith-Wiedemann Algorithm}
The algorithm above requires two computations:
\begin{itemize}
    \item Compute the rank of a sparse matrix over $\F_p$.
    \item Compute a basis for the kernel of a sparse matrix over $\F_p$ whose kernel has small dimension.
\end{itemize}
We use an implementation of the block Coppersmith--Wiedemann (BCW) algorithm. This is a black-box method that requires only the action of the operator on vectors, which we compute from the defining formula for $F$ without storing its matrix.

We will give here a sketch of the BCW algorithm. A complete treatment is beyond the scope of this paper, but a good reference is \cite{Skipness}. The BCW algorithm is based on the following general idea. Suppose $A$ is a symmetric linear operator. Pick random vectors $v_1,\dots,v_N$ in the domain of $A$. Then build the sequence of matrices for $r=0,1,2,\dots$
\begin{equation}\label{equ:wiedemann seq}
(v_i^T A^r v_j)_{i,j}.
\end{equation}
From a sufficiently long initial segment of this sequence, one computes a minimal matrix generator, a polynomial matrix describing its linear recurrence relations. This generator yields a lower bound for the rank of $A$; see \cite{Skipness} for the reconstruction and its hypotheses. We leave the details of this reconstruction and correctness proofs to \cite{Skipness}, but we remark on several features of the algorithm:
\begin{itemize}
\item To obtain the matrix sequence, we only need to implement the action of $A$ on a vector. In particular, the algorithm avoids fill-in, the appearance of new nonzero matrix entries during Gaussian elimination.
\item The BCW algorithm is a randomized algorithm. It produces a provably correct lower bound for the rank of $A$. This bound equals the rank with high probability for generic $A$, or for $A$ that is subjected to random preconditioning.
\item Our matrices need not be generic, and dense matrix--vector multiplication is too costly for these computations. We therefore use preconditioners that can be applied efficiently.
\item The BCW algorithm as presented here operates on a symmetric matrix $A$. If $A$ is not symmetric or is rectangular, one applies the algorithm instead to
\[
D_1^T A^T D_2 A D_1,
\]
with $D_1$ and $D_2$ suitably chosen preconditioning matrices. Here the least expensive choice of preconditioner is to take random diagonal matrices for $D_1$ and $D_2$. This did not suffice in our computations because the rank estimates returned by BCW were often smaller than the rank. Instead, we apply in addition a row mixing step.
\item An extension of the BCW algorithm can also compute $N$ kernel vectors. This requires a second sequence computation after the first BCW computation, to recover the kernel vectors.
\end{itemize}

For a rectangular matrix $A$, the preconditioned operator $H=D_1^T A^T D_2 A D_1$ satisfies $\rank H\leq\rank A$. Thus a guaranteed lower bound for $\rank H$ equal to the number of columns of $A$ proves injectivity of $A$. A vector $x\in\ker H$ gives a candidate $D_1x$ for $\ker A$; its residual $A(D_1x)$ must be checked directly.

For the larger matrices, generating the sequence \eqref{equ:wiedemann seq} accounts for most of the computation time. For smaller matrices, reconstructing the rank information from the sequence can take longer.

\subsection{Implementation details}
We use a custom GPU (CUDA) implementation of the sequence computation \eqref{equ:wiedemann seq} for the matrices considered. To recover rank estimates, lower bounds, and kernel vectors from these sequences, we use the CPU routines of \cite{Skipness}.

The sequence computation is matrix-free. The CUDA program applies the defining formula \eqref{equ:Fdef} without computing and storing the matrices $F_c$. We precompute and store the much smaller multiplication matrices $\mu_x$ on the CPU for use in the GPU computation.

The rank computations for the $m$ blocks can be run in parallel. We use a scheduler that uses 2 GPUs and 64 CPU cores to run the GPU and CPU parts of these computations in parallel. All experiments have been performed on servers with 64-core AMD EPYC CPUs, 1TB of RAM and two NVIDIA RTX PRO 6000 Blackwell GPUs with 96GB of RAM each.

\subsection{Computational results}
With the parameter values detailed above, the computations succeeded for even $g$ between 20 and 30, thus showing the following result.

\begin{prop}\label{prop:computational-result}
For \mbox{$g\in\{22,26,30\}$} and the parameter values specified above, the matrix $\Phi_g$ is injective over $\F_p$. Moreover, for \mbox{$g\in\{20,24,28\}$} and the deformations specified above, the matrix $\Phi_g(t)$ is injective over $\F_p(\!(t)\!)$.
\end{prop}
Table~\ref{tab:sizes} gives matrix dimensions and approximate runtimes. For $g=2m$, the computation requires $m$ BCW passes when $m$ is odd and $m+6$ when $m$ is even. There is one pass for each matrix $F_c$. For even $m$, the six additional passes comprise one to recover a basis of $\ker F_0^{(0)}$, two for each of the first-order systems \eqref{equ:SK decomp}, and one to show that the replacement matrix $F_{rep}$ has full column rank.

GPU times refer to a single BCW pass on one GPU. Total elapsed times include CPU pre- and post-processing, with the computations scheduled across two GPUs and 64 CPU cores. All timings are approximate and were measured on shared servers, where concurrent resource use by other users affects runtimes.

All computer code can be found at the following repository.

\medskip

\url{https://github.com/wilthoma/prym_green_ancillary}

\begin{table}[htbp]
  \centering
  \small
  \setlength{\tabcolsep}{4pt}
  \renewcommand{\arraystretch}{1.2}
  \resizebox{\textwidth}{!}{%
    \begin{tabular}{rccrrr}
      \toprule
      $g$ & Size of $\Phi_g$ & Size of $F_c$
      & \shortstack{BCW\\passes}
      & \shortstack{GPU time\\per pass}
      & \shortstack{Total elapsed\\time} \\
      \midrule
      20 & $413\,270\times388\,960$
         & $\{21\,886,21\,872\}\times19\,448$
         & 16 & $0.35\,\mathrm{s}$ & $1.0\,\mathrm{min}$ \\
      22 & $1\,755\,182\times1\,662\,804$
         & $83\,980\times75\,582$
         & 11 & $2.0\,\mathrm{s}$ & $2.1\,\mathrm{min}$ \\
      24 & $7\,407\,036\times7\,054\,320$
         & $\{323\,302,323\,344\}\times293\,930$
         & 18 & $22\,\mathrm{s}$ & $7.5\,\mathrm{min}$ \\
      26 & $31\,097\,794\times29\,745\,716$
         & $1\,248\,072\times1\,144\,066$
         & 13 & $5.7\,\mathrm{min}$ & $49\,\mathrm{min}$ \\
      28 & $130\,007\,500\times124\,807\,200$
         & $\{4\,828\,916,4\,828\,784\}\times4\,457\,400$
         & 20 & $1.50\,\mathrm{h}$ & $15.9\,\mathrm{h}$ \\
      30 & $541\,574\,100\times521\,515\,800$
         & $18\,721\,080\times17\,383\,860$
         & 15 & $31.6\,\mathrm{h}$ & $10.6\,\mathrm{d}$ \\
      \bottomrule
    \end{tabular}%
  }
  \caption{\label{tab:sizes}
  Matrix dimensions (rows $\times$ columns), BCW passes, and approximate runtimes for $g=2m$. Paired row counts for $F_c$ reflect the parity of $c$. Timing conventions are described in the text.}
\end{table}

\begin{proof}[Proof of Theorem~\ref{thm:main}]
For $g=22,26,30$, Proposition~\ref{prop:computational-result} supplies rational nodal Prym curves over finite fields for which $\Phi_g$ has full column rank. The pencils used in their construction have no common geometric zero by the checks in Section~\ref{sec:artinian}. For $g=20,24,28$, Proposition~\ref{prop:computational-result} gives injectivity of $\Phi_g(t)$ over $\F_p(\!(t)\!)$, equivalently of $F(t)$, and Lemma~\ref{lem:finite-field-specialization} supplies finite-field Prym curves satisfying the same hypotheses. Proposition~\ref{prop:specialization} then proves the required vanishing for a general complex Prym curve in each of the six genera.
\end{proof}

\bigskip

\end{document}